\documentclass[a4paper,oneside]{amsart}

\title[On cohomology of  crystallographic groups of split type]{On cohomology of  crystallographic groups with cyclic holonomy of split type}

\author{Nansen Petrosyan}
\address{Department of Mathematics, Catholic University of Leuven, Kortrijk, Belgium}%
\email{Nansen.Petrosyan@kuleuven-kortrijk.be}%

\author{Bartosz Putrycz}
\address{Department of Mathematics, Catholic University of Leuven, Kortrijk, Belgium}%
\email{Bartosz.Putrycz@kuleuven-kortrijk.be}%

\thanks{Both authors were supported by the Research Fund K.U.Leuven.}
\thanks{The first author was also supported by the FWO-Flanders Research Fellowship.}

\usepackage[utf8]{inputenc}
\usepackage{amsthm}
\usepackage{amssymb}
\usepackage{mathtools}
\usepackage{tikz}
\usepackage{enumerate}
\usepackage{multirow}

\numberwithin{equation}{section}

\newcommand{\divides}{|}

\theoremstyle{plain}
\newtheorem{thm}{Theorem}[section]
\newtheorem{prop}[thm]{Proposition}
\newtheorem{cor}[thm]{Corollary}

\newtheorem{conj}{Conjecture}[section]
\theoremstyle{definition}
\theoremstyle{remark}
\newtheorem{rem}{Remark}

\usepackage{calc}
\newcommand{\gen}[1]{\langle #1 \rangle}

\newcommand{\Z}{\mathbb{Z}}
\newcommand{\R}{\mathbb{R}}
\newcommand{\B}{\scriptscriptstyle B}
\newcommand{\M}{{\scriptstyle M}}

\DeclareMathOperator{\Id}{Id}
\DeclareMathOperator{\im}{im}

\DeclareMathOperator{\Hom}{Hom}
\DeclareMathOperator{\SNF}{SNF}
\DeclareMathOperator{\oH}{H}
\DeclareMathOperator{\odiag}{diag}

\begin{document}

\maketitle

\begin{abstract}
We disprove a conjecture stating that the integral cohomology of any $n$-dimensional crystallographic group $\Z^n\rtimes \Z_m$ admits a decomposition:
$$\mathrm{H}^* (\Z^n\rtimes \Z_m) \cong \bigoplus_{i+j=*} \mathrm{H}^i (\Z_m, \mathrm{H}^j (\Z^n))$$
by providing a complete list of counterexamples up to dimension 5. We also find a counterexample with odd order holonomy, $m=9$, in dimension
8 and finish the computations of the cohomology of $6$-dimensional crystallographic groups arising as orbifold fundamental groups of certain Calabi-Yau toroidal orbifolds.
\end{abstract}

\section{Introduction}
An $n$-dimensional {\it crystallographic group} $\Gamma$ is a discrete subgroup
of isometries of $\R^n$ acting properly discontinuously  and cocompactly on
$\R^n$. By the first Bieberbach theorem (see \cite{Ch}), every such group
has a normal subgroup $L$ of translations which is a uniform lattice of $\R^n$
and the holonomy group  $\Gamma/L$  is finite.

In \cite{AGPP},  there is  a complete structure theorem on the cohomology of
crystallographic groups with cyclic holomony of prime order. When such a group $\Gamma$
contains a torsion element, i.e. $\Gamma = \Z^n\rtimes \Z_p$, the theorem
asserts that the integral cohomology of  $\Gamma$ is given by  the cohomology
of $\Z_p$ with coefficients in the cohomology of the lattice $L$. Also, it is
conjectured that a similar decomposition holds for the cohomology of $\Gamma=
L\rtimes G$ for any  finite cyclic group $G$.

\begin{conj}[{\cite[5.2]{AGPP}}] \label{conj}
	Suppose that $G$ is a finite cyclic group and $L$ a finitely generated
	$\Z G$-lattice; then for any $m \geq 0$ we have
	\[
	\mathrm{H}^k (L \rtimes G, \Z) \cong \bigoplus_{i+j=k} \mathrm{H}^i (G, \mathrm{H}^j (L, \Z)).
	\]
\end{conj}

We show that the conjecture already fails for a $4$-dimensional
crystallographic groups with  holonomy $\Z_4$   (see  Corollary \ref{counter}). This is the lowest possible dimension of a crystallographic group for which the conjecture is
not true.  In Section 4, we also compute the cohomology of all $4$- and
$5$-dimensional crystallographic groups which do not satisfy the conjecture. There are 2 in dimension 4 both with holonomy
$\Z_4$  and  in dimension 5, there are 5 with holonomy $\Z_4$ and one with holonomy $\Z_8$.
As further applications of our methods, in Section 5, we finish the
computations of the cohomology of $6$-dimensional crystallographic groups
which arise as orbifold fundamental groups of certain toroidal orbifolds
discussed in {\cite[Sec.~6]{AGPP}}. Also, we give an example of  an 8-dimensional  crystallographic group with holonomy $\Z_9$ which is the first counterexample with odd order holonomy.

Our approach is straightforward, as we compute both sides of the conjectured
equation and immediately observe that they are not isomorphic. The method of
computations is based on the so-called {\it twisted tensor product}
construction introduced by Wall in \cite{Wall}. Roughly stated, given an
arbitrary group extension $1\to L\to \Gamma\to G\to 1$ and free resolutions
$B_*$ and $C_*$ of $\Z$ over $\Z L$ and $\Z G$ respectively,
by inducting $B_*$ to a resolution $\mbox{Ind}_L^{\Gamma}B_*$ over $\Z \Gamma$ and then tensoring with $C_*$ over $\Z G$, assuming trivial right action on $G$ on $\mbox{Ind}_L^{\Gamma}B_*$, one obtains an augmented chain complex of free $\Z\Gamma$-modules. Wall then proves that one can recursively construct new  differentials of the complex to obtain an acyclic complex.

This method of computing the cohomology of crystallographic groups has already
been implemented in GAP (see \cite[Package HAP]{GAP4}). In Sections 2 and 3, we discuss how  we adapt the algorithm to the case of crystallographic groups of split type. This shortens the computing times and allows us to find some counterexamples to the conjecture in dimensions 6 and 8.

Recently, using different methods, a 6-dimensional counterexample to the conjecture with holonomy $\Z_4$ has been found by Langer and L\"{u}ck (\cite[0.6]{Lu}). In fact, they show that there is a counterexample to the conjecture  for every holonomy group whose order is divisible by 4. They also verify the conjecture with an extra assumption that the action of $G$ on $L$ is free (\cite[0.5]{Lu}).

In \cite[1.2]{Pet}, it was conjecture that the Lyndon-Hochschild-Serre spectral sequence associated to $\Z^n\rtimes \Z_m$ collapses at $E_2$ not only with integral coefficients,  but more generally for  all coefficient modules $A$ that are  $\Z$-free of finite rank having trivial $\Z^n$-action. We can easily show that several of the counterexamples to Conjecture \ref{conj} are also counterexamples to this conjecture. In Section 4 (see Theorem 4.2), we present a $3$-dimensional counterexample  to Conjecture \cite[1.2]{Pet}. Interestingly, this is the lowest possible dimension for a crystallographic group with cyclic holonomy of split type whose  associated   Lyndon-Hochschild-Serre spectral sequence collapses with integral coefficients but does not collapse  for some other  coefficients $A$.

\section{The cruxes of the method}

Before presenting the main steps used in our computations, first we discuss two notions that
are essential to this method.

\subsection{Twisted tensor product} \label{sect:Wall}
Let $1 \to L \to \Gamma \to G \to 1$
be an arbitrary extension of groups. Suppose $(B_r, r\geq 0)$ and $(C_s, s\geq 0)$ are free $\Z L$ and $\Z G$-resolutions of $\Z$,
respectively and denote by $\partial_*$ the differential of $C_*$.

The induced module $\mbox{Ind}_L^{\Gamma}B_*=\Z \Gamma \otimes_{\Z L} B_*$ is free over $\Z\Gamma$.
Since  induction  is an exact functor, $\mbox{Ind}_L^{\Gamma}B_*$ with the differentials induced from those of $B_*$ becomes
  a free $\Z \Gamma$-resolution of $\Z G$.

Next, let us endow each module $\mbox{Ind}_L^{\Gamma}B_*$ with the trivial right $G$-action and define:
\begin{align*}
A_{r,s}&:= \mbox{Ind}_L^{\Gamma}B_r\otimes_{\Z G} C_s.\\
\intertext{Set $\alpha_s = \mbox{rk}_{\Z G}(C_s)$ and denote  by $\mbox{Ind}_L^{\Gamma}B$  the graded complex $\bigoplus_r \mbox{Ind}_L^{\Gamma}B_r$ and let  $\epsilon$ be its augmentation. Then}
D_s &:= \bigoplus_r A_{r,s}=\mbox{Ind}_L^{\Gamma}B \otimes_{\Z G} C_s\\
\intertext{is a direct sum of $\alpha_s$ copies of $\mbox{Ind}_L^{\Gamma}B$, which together with
augmentation $\epsilon_s:=(\Id \otimes_{\Z L} \epsilon)^{\alpha_s}$ onto $C_s$ entails a free $\Z \Gamma$-resolution of $C_s$.
Lastly, we denote by $d_0$  the differential of each complex $D_s$ and define:}
A&:= \bigoplus_s D_s =\bigoplus_{r,s} A_{r,s}
\end{align*}
 graded by $r+s$.

The following crucial result was proven in \cite{Wall}. In fact, its  proof will comprise the main steps of the algorithm which we will
discuss later.

\begin{thm}[{\cite[Lem.~2, Th.~1]{Wall}}] \label{lem:Wall}
	There exist $\Z \Gamma$-homomorphism $d_k \colon A_{r,s} \to A_{r+k-1,s-k}$ $(k \geq 1, s \geq
	k)$ such that
	\begin{enumerate}[$($i$)$]
\item $\epsilon_{s-1} d_1 = \partial \epsilon_s \colon A_{0,s} \to C_{s-1}$
\smallskip
\item $\sum_{i=0}^{k} d_i d_{k-i} = 0$, for each $k$, $($where $d_{k|A_{r,s}}$ is interpreted as zero if $r = k = 0$ or if $s < k.)$
\end{enumerate}
\smallskip
Moreover, with the differential $d=\sum_{k=0}^{\infty}d_k$, the complex
$(A,d)$ is acyclic and hence it yields a free  $\Z \Gamma$-resolution of $\Z$.
\end{thm}

\subsection{Contracting homotopies} \label{sect:contr-hom}

Let $(Q,d)$ be an acyclic chain complex. It will be often
necessary to take preimages of $d$ for elements which
are in $\ker d=\im d$. A suitable
computational method for this is by using a contracting homotopy.
More about this approach could be found in \cite[Section~3]{El1}.

A {\it contracting homotopy} of an acyclic complex $Q$ is a chain map $h \colon
Q_i \to Q_{i+1}$ such that $hd + dh = \Id$.  Then for each  $y \in \ker d$, we have $dh(y) = y$. So $h$ maps such an element $y$ to its
preimage under $d$.

Contracting homotopies are often easy to construct.
To obtain a contracting homotopy for a $\Z[\Z^n]$-resolution $B$ of $\Z$,
we will use the standard formula given in \cite[p. 214]{CE},
which provides a contracting homotopy for a tensor
product of acyclic complexes equipped with contracting homotopies.

Let $L=\Z^n$ and $\Gamma=L\rtimes G$. We  need to explain how to define a contracting homotopy on the induced complex
$\mbox{Ind}_L^{\Gamma}B$ from a given contracting homotopy $h$ on  $B$. Since every element of $\mbox{Ind}_L^{\Gamma}B$ can be written as a direct sum of elements of the form $(1, g) \otimes_{\Z L} y$ for $(1, g)\in \Gamma$ and $y\in B$, we define the contracting homotopy  by:
	\[
	f:\mbox{Ind}_L^{\Gamma}B\to \mbox{Ind}_L^{\Gamma}B,	\quad (1, g)\otimes_{\Z L} y\mapsto (1, g) \otimes_{\Z L} h(y).
	\]

\subsection{The steps involved}\label{steps}

We are now ready to describe the key steps of the algorithm used to compute the cohomology of an $n$-dimensional crystallographic
group $\Gamma=L\rtimes G$. The reader may find it helpful to refer to the next section where we explicitly implement these steps in a specific example.

To obtain the free $\Z L$-resolution $B_*$ of $\Z$, we  tessellate $\R^n$ into standard $n$-cubes of length 1. This defines an
$L$-equivariant CW-structure on $\R^n$ and the associated chain complex yields the desired resolution.

We denote by $t_i$ for $1\leq i \leq n$ the generators of $L=\Z^n$ which correspond to translations by 1 in the coordinate $i$.
We denote by $e$ the origin of $\R^n$, by $e_i$   for $1 \leq i \leq n$ the
	1-dimensional segment from $e$ to $t_ie$, and  by $e_{i_1i_2\ldots i_m}$ the $m$-dimensional cube spanned by $e_{i_1},
	e_{i_2} \ldots e_{i_m}$. Then
	\[B_m = \gen{e_{i_1\ldots i_m}, 1 \leq i_1 < \cdots < i_m \leq n}_{\Z L}
		\quad \text{ for }\quad 0 \leq m \leq n
		%\\
	%\end{align*}
	\]
and the differentials of $B_*$, denoted by $d^{\B}_*$, are given by:
	\[
	d^{\B}_m(e_{i_1\ldots i_m}) = \sum_{j=1}^{m} (-1)^{j-1} (t_{i_j} - 1)
	e_{i_1\ldots \hat{i_j}\ldots i_m}.
	\]

Next, we need a free $\Z G$-resolution $C_*$ of $\Z$.

\begin{rem}\label{rem1}{\rm In our computations, the holonomy  will always be a finite
	cyclic group, i.e. $G=\langle x \; | \; x^q=1\rangle$. In this case, we will
	take for $C_*$  the standard $2$-periodic resolution   $C_i=\Z G$ for all $i\geq 0$
	and $\partial_{i+1}: C_{i+1} \xrightarrow{x-1} C_{i}$ when $i$ is even and
	$\partial_{i+1}: C_{i+1} \xrightarrow{x^{q-1}+\cdots + x+1} C_{i}$ when $i$
	is odd.}
\end{rem}

Now, using twisted tensor product construction, we obtain a free $\Z \Gamma$-resolution $(A,d)$ of $\Z$ as follows:

\begin{enumerate}[(i)]
\item As discussed, we construct the resolutions $(C_*, \partial)$ and $(\mbox{Ind}_L^{\Gamma}B_*, \mbox{Id}\otimes_{\Z L}d^{\B}_*)$, and free $\Z\Gamma$-modules $A_{r,s}$ for each $0\leq r\leq n$ and $0\leq s$ and set $A_m=\bigoplus_{r+s=m} A_{r,s}$.
\smallskip
\item For $n=1$, we define a contracting homotopy ${}^1h$ on the $\Z[\Z]$-free resolution $(B_*, d^{\B}_*, n=1)$
%:$0\to\Z[\Z]\langle e_1\rangle\xrightarrow{t_1-1}\Z[\Z]\langle e\rangle \xrightarrow{\epsilon} \Z\to 0$
 by ${}^1h(1)=e$ and
\[
	{}^1h(t_1^je)=
	\begin{cases}
		\quad \sum_{i=0}^{j-1} t_1^i e_{i_1\ldots i_m} &  j >
		0\\
		- \sum_{i=1}^{-j} t_1^{-i} e_{i_1\ldots i_m} & \text j < 0 \\
		0 &  j=0.
	\end{cases}
	\]
For each $k\geq 1$, since $(B_*, d^{\B}_*, n=k+1)$ is isomorphic to the tensor product of $(B_*, d^{\B}_*, n=k)$ and the above resolution, we can and will  define a contracting homotopy $h:B\to B$ by the recursive formula (see \cite[p. 214]{CE}):
$${}^{k+1}h={}^{k}h\otimes \iota + ({}^{k}h\epsilon)\otimes {}^1h,$$ where $\iota$ is the identity map on $(B_*, d^{\B}_*, n=1)$ and ${}^{k}h$ is the
contracting homotopy on $(B_*, d^{\B}_*, n=k)$.

\smallskip
\item Let $r=0$ and $\beta$ be a generator of $A_{0,s}$.
%Since $\epsilon_{s-1}$ is an epimorphism,
We define $d_1(\beta)=f(\partial(\epsilon_s(\beta)))\subseteq A_{0,s-1}$.
%We let $d_1(\beta)=h_{-1,s-1}(\partial(\epsilon_s(\beta)))$.
\end{enumerate}
\smallskip

\noindent For $r = 1$, we have that
$\epsilon_{s-1} d_1 d_0 = \partial \epsilon_s d_0 = 0$. Hence, $d_1 d_0:A_{1,s}\to A_{0,s-1}$ maps into
$\ker \epsilon_s = \im d_0$.
\smallskip

\begin{enumerate}[(i)]
\setcounter{enumi}{3}
\item So, for any generator $\beta\in A_{1,s}$,
we define $d_1(\beta)=-f(d_1(d_0(\beta)))$.
Similar occurs for $r \geq 2$ and
for any generator $\beta\in A_{2,s}$, we define
$d_1(\beta)=-f(d_1(d_0(\beta)))$.% replacing $\epsilon$ by $d_0$.

\end{enumerate}
\smallskip
\noindent  For $k \geq 2$ we need to define $d_k$
which satisfy the equation $\sum_{i=0}^k d_i d_{k-i} = 0$.
Suppose, we defined $d_i$ for $i<k$ and $d_k|A_{r-1,s}$ satisfying this
property. It is not difficult to check that  $\sum_{i=1}^{k} d_i d_{k-i}$  % \colon A_{r,s} \to A_{r+k-2,s-k}$
 is in $\ker d_0 = \im d_0$ (see Lemma 2 of \cite{Wall}).
\smallskip
\begin{enumerate}[(i)]
\setcounter{enumi}{4}
\item Then, for a generator $\beta\in A_{r,s}$ we take $d_k(\beta)=-f( \sum_{i=1}^{k} d_i d_{k-i})(\beta)$.
\end{enumerate}
\smallskip
\noindent This yields the free $\Z \Gamma$-resolution $(A,d)$. To calculate the cohomology of $\Gamma$ we:
\smallskip
\begin{enumerate}[(i)]
\setcounter{enumi}{5}
\item apply the functor $\mbox{Hom}_{\Z\Gamma}(-,\Z)$  to $(A,d)$ to obtain a cochain complex of finitely generated $\Z$-free modules $(F, \delta)$.
\smallskip
\item For each $0\leq i\leq n+1$, reduce the matrix representing the boundary map $\delta_i:F_i\to F_{i+1}$ to Smith normal form and read off the cohomology group $\mathrm{H}^{i+1}(\Gamma)$ via the isomorphism: \[F_{i+1}/\mbox{Im}\delta_i\cong \mathrm{H}^{i+1}(\Gamma)\oplus \mbox{Im}\delta_{i+1}.\]
      \end{enumerate}
\smallskip
\begin{rem}\label{rem2}{\rm When $G$ is a cyclic group and $C_*$ is its standard 2-periodic resolution, since the resolution $(B, d^{\B})$ has length $n$,  one can easily observe  that
	the resolution $(A,d)$ will also be 2-periodic starting from dimension $n+1$. So, in all the steps we can stop the computaions once we reach this dimension.}
\end{rem}

\section{A counterexample} \label{sect:counter}
In this section, we provide a counterexample to Conjecture \ref{conj},
by applying the computational steps of Section 2.

	Let $\Gamma$ be a $4$-dimensional crystallographic group $L \rtimes G$ where $G=\langle \M \;|\; \M ^4=1\rangle$ is the
	cyclic of order 4  acting on $L =\Z^4$  by a left multiplication given  by the matrix:
\smallskip

\begin{equation} \label{counter:matrix}
M = {\footnotesize\begin{bmatrix}
 0 & 1 & 0 & 0 \\
-1 & 0 & 0 & 1 \\
 0 & 0 & -1 & 1 \\
 0 & 0 & 0 & 1
\end{bmatrix}}\end{equation}
\smallskip

\begin{prop}
	The integral cohomology of $\Gamma$ is as follows:
\smallskip
\begin{eqnarray*}
\mathrm{H}^i(\Gamma) &=& \left\{ \begin{array}{lll}
 \Z  & i=1\\
	   \Z \oplus \Z_4 \oplus \Z_2  & i=2\\
		    \Z \oplus \Z_4 \oplus \Z_2  & i=3\\
		  \Z_4^2   & i=2k, & k\geq 2\\
	 \Z_2^4   & i=2k+1, & k\geq 2
                   \end{array}\right.
%                    \[\mbox{H}^1(\Gamma) &=& \left\{ \begin{array}{ccc}
%                    & \Z \\
%                    1 &  &  \mbox{otherwise}
%                   \end{array}\right. \\
\end{eqnarray*}
\end{prop}
\smallskip
\begin{proof} Let $(B,d^{\B})$ be the  free $\Z L$-resolution  of $\Z$ defined  in \ref{steps} for $n=4$.

 To simplify the notation we introduce the symbol:
	\[
	\mathfrak{C}(j, t_k, e_{i_1\ldots i_m}) :=
	\begin{cases}
		\quad \sum_{i=0}^{j-1} t_k^i e_{i_1\ldots i_m} &  j >
		0\\
		- \sum_{i=1}^{-j} t_k^{-i} e_{i_1\ldots i_m} & \text j < 0 \\
		0 &  j=0
	\end{cases}
	\]
	Then, the contracting homotopy $h\colon B \to B$ of the augmented resolution $B$  (where we set $B_{-1}=\Z$)
	is given by:
\begin{center}\begin{align*}	h(1) &= e, \\%, \Z[G]\text{ ? linearity vs G-equiv}
		%\\\\
h(t^i_1t^j_2t^k_3t^l_4 e) & =
		t_2^jt_3^kt_4^l \mathfrak{C}(i, t_1, e_1) + t_3^kt_4^l \mathfrak{C}(j, t_2, e_2)  \\
		& \quad + t_4^l \mathfrak{C}(k, t_3, e_3) + \mathfrak{C}(l, t_4, e_4), \\	
	\begin{split}
		h(t^i_1t^j_2t^k_3t^l_4 e_1) & = 0
	\end{split}, \\
	\begin{split}
		h(t^i_1t^j_2t^k_3t^l_4 e_2) & = t_2^jt_3^kt_4^l \mathfrak{C}(i,t_1,e_1e_2),
	\end{split} \\
	\begin{split}
		h(t^i_1t^j_2t^k_3t^l_4 e_3) & =
		t_2^jt_3^kt_4^l \mathfrak{C}(i, t_1, e_1e_3) + t_3^kt_4^l \mathfrak{C}(j, t_2, e_2e_3)
	\end{split}, \\
	\begin{split}
	h(t^i_1t^j_2t^k_3t^l_4 e_4) & =
		t_2^jt_3^kt_4^l \mathfrak{C}(i, t_1, e_1e_4) + t_3^kt_4^l \mathfrak{C}(j, t_2, e_2e_4), \\
		&\quad  + t_4^l \mathfrak{C}(k, t_3, e_3e_4),
	\end{split}\\
\begin{split}
h(t^i_1t^j_2t^k_3t^l_4 e_1e_2) & = h(t^i_1t^j_2t^k_3t^l_4 e_1e_3) = h(t^i_1t^j_2t^k_3t^l_4 e_1e_4) = 0 \\
		h(t^i_1t^j_2t^k_3t^l_4 e_2e_3) & = t_2^jt_3^kt_4^l \mathfrak{C}(i, t_1, e_1e_2e_3), \\
		h(t^i_1t^j_2t^k_3t^l_4 e_2e_4) & = t_2^jt_3^kt_4^l \mathfrak{C}(i, t_1, e_1e_2e_4), \\
		h(t^i_1t^j_2t^k_3t^l_4 e_3e_4) & =
		t_2^jt_3^kt_4^l \mathfrak{C}(i, t_1, e_1e_3e_4)\\
		&\quad + t_3^kt_4^l \mathfrak{C}(j, t_2, e_2e_3e_4),
\end{split}\\
	    h(t^i_1t^j_2t^k_3t^l_4 e_1e_2e_3) & = h(t^i_1t^j_2t^k_3t^l_4 e_1e_2e_4) = h(t^i_1t^j_2t^k_3t^l_4 e_1e_3e_4) = 0,\\
h(t^i_1t^j_2t^k_3t^l_4 e_2e_3e_4) & = t_2^jt_3^kt_4^l \mathfrak{C}(i, t_1, e_1e_2e_3e_4). \\
	\end{align*}\end{center}

Next, we set $x = (1, \M)\in L\rtimes G$
and consider the standard $\Z G$-resolution $(C,\partial)$ defined in Remark \ref{rem1}:
	\begin{equation} \label{Z4:res}
		\cdots \rightarrow \Z[x]/(x^4-1) \xrightarrow{x^3+x^2+x+1}\Z[x]/(x^4-1)\xrightarrow{x-1}\Z[x]/(x^4 - 1) \rightarrow
	\Z \rightarrow 0.
	 \end{equation}
Since, every $C_i$ is 1-generated as a $\Z G$-module,  by construction: $$D_s =\bigoplus_r {A_{r,s}}\cong\mbox{Ind}_L^{\Gamma}B$$  for each $s\geq 0$.

Occasionally, we will add a superscript to generators of $B$ to denote
	to which $A_{r,s}$ they belong, i.e.
	\[
	1 \otimes_L e^s_{i_1i_2\ldots i_r} \in A_{r,s} \quad \text{ for } r,s \geq 0, 1
	\leq i_1 < \ldots < i_r \leq 4
	\]
	and from now on, we will simplify the notation by setting:
	\[
%xte^s_{i_1\ldots i_r} := xt \otimes_L e^s_{i_1\ldots i_r} = x \otimes_L te^s_{i_1\ldots i_r}
	ge^s_{i_1\ldots i_r} := g \otimes_L e^s_{i_1\ldots i_r}.
	\]

Now, using the recursive steps (iii)-(v) discussed in Section 2.3, we  compute the maps $d_k \colon A_{r,s} \to A_{r+k-1,
	s-k}$ of Theorem \ref{lem:Wall}.

Let $s\geq 1$. Then,

	\[
	\begin{split}
		d_1(e^{2s-1}) & = (x-1)e \\
		d_1(e^{2s}) & = (x^3+x^2+x+1)e \\
	\end{split}
	\]

	\[
	\begin{split}
		d_1(e_1^{2s-1}) & = e_1 - x e_2 \\
		d_1(e_2^{2s-1}) & = x t_1^{-1}t_4e_1 + e_2 - x e_4 \\
		d_1(e_3^{2s-1}) & = x t_3^{-1}t_4 e_3 + e_3 - x e_4 \\
		d_1(e_4^{2s-1}) & = -x e_4 + e_4 \\
	\end{split}
	\]
	\[
	\begin{split}
		d_1(e_1^{2s}) & =
		x^2 t_1^{-1}t_4 e_1  - e_1 + x^3 t_2^{-1}t_4 e_2 - x e_2 - x^3 e_4 - x^2 e_4 \\
		d_1(e_2^{2s}) & =
		- x^3 e_1 + x t_1^{-1}t_4 e_1 + x^2 t_2^{-1}t_4 e_2 - e_2 - x^2 e_4 - x e_4 \\
		d_1(e_3^{2s}) & =
		 x^3 t_3^{-1}t_4 e_3-x^2 e_3 + x t_3^{-1}t_4 e_3 - e_3 -x^3 e_4 -x e_4 \\
		d_1(e_4^{2s}) & = -(x^3 + x^2 + x + 1) e_4 \\
	\end{split}
	\]
	\[
	\begin{split}
		d_1(e_{12}^{2s-1}) & =
		x t_1^{-1}t_4 e_{12} - e_{12}  + x e_{24}\\
		d_1(e_{13}^{2s-1}) & =
		 - e_{13} - x t_3^{-1}t_4 e_{23} + x e_{24}
		\\
		d_1(e_{14}^{2s-1}) & =
		 - e_{14} + x e_{24}
		\\
		d_1(e_{23}^{2s-1}) & =
		x t_1^{-1}t_3^{-1}t_4^2 e_{13} - x t_1^{-1}t_4 e_{14} - e_{23} + x t_3^{-1}t_4 e_{34}
		\\
		d_1(e_{24}^{2s-1}) & =
		- x t_1^{-1}t_4 e_{14} - e_{24}
		\\
		d_1(e_{34}^{2s-1}) & =
		-x t_3^{-1}t_4 e_{34} - e_{34}
		\\
	\end{split}
	\]
	\[
	\begin{split}
		d_1(e_{12}^{2s}) & =
		  x^3 t_2^{-1}t_4 e_{12}
		+ x^2 t_1^{-1}t_2^{-1}t_4^2 e_{12}
		+ x t_1^{-1}t_4 e_{12}
		+ e_{12} \\
		& \quad
		- x^3  e_{14}
		- x^2 t_1^{-1}t_4 e_{14}
		+ x^2 t_2^{-1}t_4 e_{24}
		+ x e_{24}
		\\
		d_1(e_{13}^{2s}) & =
		- x^2 t_1^{-1}t_4 e_{13}
		+ e_{13}
		+ x^3 t_2^{-1}t_3^{-1}t_4^2 e_{23}
		- x t_3^{-1}t_4 e_{23} \\
		& \quad
		- x^3 t_2^{-1}t_4 e_{24}
		+ x e_{24}
		+ x^3 t_3^{-1}t_4 e_{34}
		- x^2 e_{34}
		\\
		d_1(e_{14}^{2s}) & =
		- x^2 t_1^{-1}t_4 e_{14}
		+  e_{14}
		- x^3 t_2^{-1}t_4 e_{24}
		+ x  e_{24}
		\\
		d_1(e_{23}^{2s}) & =
		- x^3 t_3^{-1}t_4 e_{13}
		+ x t_1^{-1}t_3^{-1}t_4^2 e_{13}
		+ x^3 e_{14}
		- x t_1^{-1}t_4 e_{14} \\
		& \quad
		- x^2 t_2^{-1}t_4 e_{23}
		+ e_{23}
		- x^2 e_{34}
		+ x t_3^{-1}t_4 e_{34}
		\\
		d_1(e_{24}^{2s}) & =
		- x^2 t_2^{-1}t_4 e_{24}
		+ e_{24}
		+ x^3 e_{14}
		- x t_1^{-1}t_4 e_{14}
		\\
		d_1(e_{34}^{2s}) & =
		- x^3 t_3^{-1}t_4 e_{34}
		+ x^2 e_{34}
		- x t_3^{-1}t_4 e_{34}
		+ e_{34}
		\\
	\end{split}
	\]

	\[
	\begin{split}
		d_1(e^{2s-1}_{123}) & =
		  x t_1^{-1}t_3^{-1}t_4^2 e_{123}
		+ e^{2s-2}_{123}
		- x t_1^{-1}t_4 e_{124}
		- x t_3^{-1}t_4 e_{234}
		\\
		d_1(e^{2s-1}_{124}) & =
		- x t_1^{-1}t_4 e_{124}
		+ e_{124}
		\\
		d_1(e^{2s-1}_{134}) & =
		  e_{134}
		+ x t_3^{-1}t_4 e_{234}
		\\
		d_1(e^{2s-1}_{234}) & =
		- x t_1^{-1}t_3^{-1}t_4^2 e_{134}
		+ e_{234}
		\\
	\end{split}
	\]

	\[
	\begin{split}
		d_1(e^{2s}_{123}) & =
		  x^3 t_2^{-1}t_3^{-1}t_4^2 e_{123}
		- x^2 t_1^{-1}t_2^{-1}t_4^2 e_{123}
		+ x t_1^{-1}t_3^{-1}t_4^2 e_{123}
		- e_{123} \\
		& \quad
		- x^3 t_2^{-1}t_4 e_{124}
		- x t_1^{-1}t_4 e_{124}
		+ x^3 t_3^{-1}t_4 e_{134}
		- x^2 t_1^{-1}t_4 e_{134} \\
		& \quad
		+ x^2 t_2^{-1}t_4 e_{234}
		- x t_3^{-1}t_4 e_{234}
		\\
		d_1(e^{2s}_{124}) & =
		- x^3 t_2^{-1}t_4 e_{124}
		- x^2 t_1^{-1}t_2^{-1}t_4^2 e_{124}
		- x t_1^{-1}t_4 e_{124}
		- e_{124}
		\\
		d_1(e^{2s}_{134}) & =
		  x^2 t_1^{-1}t_4 e_{134}
		- e^{2s-1}_{134}
		- x^3 t_2^{-1}t_3^{-1}t_4^2 e_{234}
		+ x t_3^{-1}t_4 e_{234}
		\\
		d_1(e^{2s}_{234}) & =
		  x^3 t_3^{-1}t_4 e_{134}
		- x t_1^{-1}t_3^{-1}t_4^2 e_{134}
		+ x^2 t_2^{-1}t_4 e_{234}
		- e_{234}
		\\
	\end{split}
	\]

	\[
	\begin{split}
		d_1(e^{2s-1}_{1234}) & =
		(- x t_1^{-1}t_3^{-1}t_4^2 % e^{2s-2}_{1234}
		- 1)  e_{1234} \\
		d_1(e^{2s}_{1234}) & =
		(- x^3 t_2^{-1}t_3^{-1}t_4^2 % e^{2s-1}_{1234}
		+ x^2 t_1^{-1}t_2^{-1}t_4^2 % e^{2s-1}_{1234}
		%\\
		%& \quad
		- x t_1^{-1}t_3^{-1}t_4^2 % e^{2s-1}_{1234}
		+ 1 ) e_{1234}
	\end{split}
	\]

\smallskip

\noindent For $d_2$ we obtain:
	\[
		d_2(e^{s}) = 0
	\]

	\[
	\begin{split}
		d_2(e^{2s}_1) & = e_{14} \\
		d_2(e^{2s}_2) & = x^3 e_{14} \\
		d_2(e^{2s}_3) & = x^2 e_{34} + e_{34} \\
		d_2(e^{2s}_4) & = 0 \\
	\end{split}
	\]

	\[
	\begin{split}
		d_2(e^{2s+1}_i) & = 0 \quad \text{ for } i = 1,4\\
		d_2(e^{2s+1}_2) & = x^3 e_{14} + e_{24} \\
		d_2(e^{2s+1}_3) & = x^2 e_{34} + e_{34} \\
	\end{split}
	\]

	\[
	\begin{split}
		d_2(e^{2s}_{i4}) & = 0 \quad \text{ for } i=1,2,3 \\
		d_2(e^{2s}_{12}) & =
		- x^3 t_2^{-1}t_4 e_{124}
		- e_{124}
		\\
		d_2(e^{2s}_{13}) & =
		  x^2 t_1^{-1}t_4 e_{134}
		- t_4 e_{134}
		- e_{134}
		\\
		d_2(e^{2s}_{23}) & =
		  x^3 t_3^{-1}t_4 e_{134}
		+ x^2 t_2^{-1}t_4 e_{234}
		- e_{234}
		\\
	\end{split}
	\]
	\[
	\begin{split}
		d_2(e^{2s+1}_{i4}) & = 0 \quad \text{ for } i=1,2,3 \\
		d_2(e^{2s+1}_{12}) & =
		- x^3 t_2^{-1}t_4 e_{124}
		- e_{124}
		\\
		d_2(e^{2s+1}_{13}) & =
		x^2 t_1^{-1}t_4 e_{134}
		- e_{134}
		\\
		d_2(e^{2s+1}_{23}) & =
		  x^3 t_3^{-1}t_4 e_{134}
		+ x^2 t_2^{-1}t_4 e_{234}
		- t_4 e_{234}
		- e_{234}
		\\
	\end{split}
	\]
	\[
	\begin{split}
		d_2(e^{s+1}_{123}) & =
		(- x^3 t_2^{-1}t_3^{-1}t_4^2 % e_{1234}
		+ x^2 t_1^{-1}t_2^{-1}t_4^2 % e_{1234}
		+ t_4 % e_{1234}
		+ 1)e_{1234}
		\\
		d_2(e^{s}_{i_1i_2i_3}) & = 0 \quad \text{ for } (i_1,i_2,i_3)\neq(1,2,3)
		\\
	\end{split}
	\]
	\[
		d_2(e^s_{1234}) = 0
	\]
	\[
		d_k \equiv 0 \quad \text{ for } k \geq 3.
	\]

Applying the functor $\Hom_{\Z \Gamma}(-,\Z)$ to  the resolution $(A,d)$, we obtain a complex $(F, \delta)$ with dimensions:
\begin{eqnarray*}
\dim(F_i) &=& \left\{ \begin{array}{lll}
1 & i =0\\
5 & i =1\\
11 & i =2\\
15 & i = 3\\
16  & i \geq 4.\end{array}\right.
\end{eqnarray*}
\noindent After numbering the generators of $F$ in lexicographical order, we determine the matrices representing the differentials and reduce them to Smith normal form (SNF):

\begin{center}
%\begin{table}[h]
%	\caption{Smith normal forms of the matrices representing the coboundaries.
%}	
	\label{tab:SNF}
%\leftmargin{-1in}
\begin{tabular}{|c|c|}%c|}
	\hline
	% & Dimensions
	 & Diagonal of SNF  \\
	\hline
	\hline
$\delta_1$
%& $1 \times 5$
& $[0]$ \\
	\hline
$\delta_2$
%& $5 \times 11$
& $[1,1,2,4,0]$ \\
	\hline
$\delta_3$
%& $11 \times 15$
& $[1,1,1,1,2,4,0,0,0,0,0]$ \\
		\hline
$\delta_4$
%& $15 \times 16$
& $[1,1,1,1,1,1,4,4,0,0,0,0,0,0,0]$\\
	\hline
$\delta_{2k-1}$, $k\geq 3$
%& $16 \times 16$
&  $[1,1,1,1,2,2,2,2,0,0,0,0,0,0,0,0]$ \\
	\hline
$\delta_2k$, $k\geq 3$
%& $16 \times 16$
& $[1,1,1,1,1,1,4,4,0,0,0,0,0,0,0,0]$ \\
	\hline
\end{tabular}
\end{center}
%\end{table}\\

\noindent Using step (vii), we finish the computations of the cohomology of $\Gamma$.\end{proof}

Next, we compute the right hand side of the conjectured equation in \ref{conj} for the group $\Gamma$.
\begin{prop} The following holds.
\[
\oH^i(G, \oH^j(L, \Z)) =
\begin{cases}
	\Z &  0 \le j \le 3, i=0 \\
	\Z_2 & j=1,3, i \ge 1 \\
	\Z_2 &  j = 2, i \ge 1, 2 \divides i \\
	\Z_2 &  j = 4, 2 \not\divides i \\
	\Z_4 &   j = 0, i \ge 1, 2 \divides i \\
	\Z_4 &   j = 2, 2 \not\divides i\\
	0 & \text{otherwise}
\end{cases}
\]
\end{prop}
%Sheet version (to remove, just to see the pattern)
%
%\begin{center}
%%\[
%%\begin{align*}
%\begin{tabular}{c|ccccc}
%%\begin{split}
%4 &    0 & $\Z_2$ &      0 & $\Z_2$ &      0 \\
%3 & $\Z$ & $\Z_2$ & $\Z_2$ & $\Z_2$ & $\Z_2$ \\
%2 & $\Z$ & $\Z_4$ & $\Z_2$ & $\Z_4$ & $\Z_2$ \\
%1 & $\Z$ & $\Z_2$ & $\Z_2$ & $\Z_2$ & $\Z_2$ \\
%0 & $\Z$ &      0 & $\Z_4$ &      0 & $\Z_4$ \\
%\hline
%j/i & 0  & 1 & 2 & 3 & 4 \\
%%\end{split}
%\end{tabular}
%%\end{align*}
%%\]
%\end{center}
%

\begin{proof}
	Note that $\mathrm{H}^1(L, \Z) \cong \mbox{Hom}(L, \Z) \cong \Z^4$. Let it be generated by $t_i, 1 \leq i \leq 4$.
	We interpret $\mathrm{H}^j(L, \Z)$ as $j$-th exterior power $\Lambda^j(\mathrm{H}^1(L, \Z))$
	with generators $t_{i_1\ldots i_j} := t_{i_1} \wedge \ldots
	\wedge t_{i_j}$ for $1 \leq i_1 < \cdots < i_j \leq 4$.

	The action of $G$ on
	$\oH^j(L, \Z)$
	is given by:
	\[
	 t_{i_1 \ldots i_j} \cdot \M  = t_{i_1} M^T \wedge  t_{i_2} M^T
	\wedge \ldots  \wedge   t_{i_j} M^T
	\]

\noindent Arranging generators of $\oH^j(L, \Z)$ in lexicographical order,
	 we obtain the following matrices for the action
	of $\M$ on $\mathrm{H}^j(L, \Z)$:
	\[
	\begin{split}
\smallskip
	j = 1 : &
	\footnotesize{\begin{bmatrix*}[r]
   0 & -1 &  0 &  0 \\
   1 &  0 &  0 &  0 \\
   0 &  0 & -1 &  0 \\
   0 &  1 &  1 &  1 \\
	\end{bmatrix*}} \\
	j = 2 : &
	\footnotesize{\begin{bmatrix*}[r]
   1  &  0  &  0  &  0  &  0  &  0 \\
   0  &  0  &  0  &  1  &  0  &  0 \\
   0  &  0  &  0  & -1  & -1  &  0 \\
   0  & -1  &  0  &  0  &  0  &  0 \\
   1  &  1  &  1  &  0  &  0  &  0 \\
   0  &  0  &  0  &  1  &  0  & -1 \\
	\end{bmatrix*}} \\
\smallskip
	j = 3 : &
	\footnotesize{\begin{bmatrix*}[r]
  -1 &  0 &  0 &  0 \\
   1 &  1 &  0 &  0 \\
   0 &  0 &  0 &  1 \\
   1 &  0 & -1 &  0 \\
	\end{bmatrix*}} \\
\smallskip
	j = 4 : &
	 \footnotesize{\begin{bmatrix}
		-1
	\end{bmatrix}}
	\end{split}
	\]

\noindent Applying
	$\Hom_{\Z G}(-, \oH^j(L, \Z))$ to the resolution \eqref{Z4:res} for $G$, we
	obtain a complex:
	\[
	0 \to \oH^j(L, \Z) \xrightarrow{M-I} \oH^j(L, \Z) \xrightarrow{M^3+M^2+M+I}
	\oH^j(L, \Z) \xrightarrow{M-I} \cdots
	\]

with the	corresponding Smith normal forms:
	\[
	\begin{split}
		j = 1 : & \SNF(M-I) = \odiag([1,1,2,0])_{4 \times 4}, \\
		& \SNF(M^3+M^2+M+1) = \odiag([2,0,0,0])_{4 \times 4} \\
	j = 2 : & \SNF(M-I) = \odiag([1,1,1,1,4,0])_{6 \times 6}, \\
	& \SNF(M^3+M^2+M+1) = \odiag([2,0,0,0,0,0])_{6 \times 6} \\
	j = 3 : & \SNF(M-I) = \odiag([1,1,2,0])_{4 \times 4}, \\
	& \SNF(M^3+M^2+M+1) = \odiag([2,0,0,0])_{4 \times 4} \\
	j = 4 : & \SNF(M-I) = [2], \\
	& \SNF(M^3+M^2+M+1) = [0] \\
	\end{split}
	\]
Proceeding as in step (vii), we finish the computations.
\end{proof}

From the two propositions, we immediately obtain:
\begin{cor}\label{counter}  For the crystallographic group $\Gamma$, we have:
	\[
	\oH^4(\Gamma, \Z) = \Z^2_4  \neq
	\Z_4 \oplus\Z^3_2 = \bigoplus_{i+j=4} \oH^i (G, \oH^j (L, \Z))
	\]
	Therefore, Conjecture \ref{conj} is false.
\end{cor}

\section{All counterexamples up to dimension 5} \label{sect:4and5}
The algorithm for twisted tensor product is implemented, for example, in the HAP
package in the system GAP (see \cite{GAP4}). We implemented our version of the algorithm
which is adjusted to our case and allows for more efficient computations.

In this section we list all cases of crystallographic groups of dimensions up
to 5 which do not satisfy Conjecture \ref{conj}.
For the list of all crystallographic groups in low dimensions we use the
classification given in CARAT (see \cite{CARAT}).

Up to dimension 3,
%we have 15 non-isomorphic integral representations of cyclic groups.
all crystallographic groups of the form $L \rtimes G$ with $G$ being cyclic satisfy
the conjecture.

In dimension 4, there are 44 non-isomorphic
crystallographic groups of this type.
Among these, 2 do not satisfy  Conjecture \ref{conj}.
Both of them have the holonomy group of order 4.

\begin{rem}{\label{diff}}The holonomy representation of the first group is generated by the matrix:
{\footnotesize
\[
\begin{bmatrix*}[r]
 -1 &  0 &  0 &  0 \\
  0 &  0 &  0 & -1 \\
  0 &  1 &  0 &  1 \\
  0 &  0 & -1 &  1 \\
\end{bmatrix*}
\]}which has the cohomology given in the Table \ref{tab:homology} under the
notation $1:4\rtimes 4$. We observe that:
\[
\begin{split}
H^4(\Z^4 \rtimes \Z_4, \Z) & = \Z_4 \oplus \Z_2^3, \\
\bigoplus_{i+j=4} H^i(G, H^j(\Z^4, \Z)) & = \Z_4^2\oplus \Z_2^2.
\end{split}
\]
This implies that, in the associated Lyndon-Hochschild-Serre spectral sequence, there are nonzero differentials.\end{rem}
\begin{rem}
Let us note that  the holonomy representation of this group is $\Z$-equivalent to a direct product
of representations of dimensions 1 and 3, and the 3-dimensional representation
is the example $\rho_6$ from \cite[Section 5]{AGPP}. It was not known if there
was a special free $\Z [\Z ^3]$-resolution of
$\Z$ that admitted a compatible action of $\Z_4$ via the representation $\rho_6$ (see \cite[2.4, 5.1]{AGPP}). The example of the group $1:4\rtimes 4$ shows
such a compatible action can never exist. Otherwise, by Lemma 2.2  and Theorem 2.3 in \cite{AGPP}, we would arrive at a contradiction.  In fact, we can say more:
\end{rem}

\begin{thm}{\label{genconj}} Consider $L=\Z^3$ and $\Gamma=L\rtimes_{\rho_6} \Z_4$. Let $A=L$ as a $\Z\Gamma$-module via the representation $\rho_6$. Then, in the Lyndon-Hochschild-Serre spectral sequence associated to $\Gamma$,  the differential $d^{0,2}_2(A):E_2^{0,2}(A)\to E_2^{2,1}(A)$ is nonzero. In particular,
$$\oH^2(\Gamma, A)\not\cong  \bigoplus_{i+j=2}\oH^i(\Z_4, \oH^j(L, A)).$$
\end{thm}

The group $\Gamma$ gives the lowest possible counterexample to a more general form of Conjecture {\ref{conj}} stated in \cite[1.2]{Pet}, where one allows nontrivial coefficients. This is because  any group $\Z^n\rtimes \Z_m$ for $n\leq 2$ admits a local compatible action (see \cite[3.1]{AP}) and therefore, by a slight generalization of a theorem of Adem and Pan (see the proof of \cite[2.3]{AP}), satisfies this more general form of the conjecture.

\begin{proof}[Proof of \ref{genconj}] We apply the theory of characteristic classes introduced by Sah in \cite{Sah} and further studied in \cite{Pet} and  \cite{DP}.

Suppose, by a way of contradiction that $d^{0,2}_2(A)=0$. One can easily check that  $\oH_2(A,\Z)$ and $\oH_3(A,\Z)$, as $\Z\Gamma$-modules, are isomorphic  to $A$  and the trivial module $\Z$, respectively.

Now, the characteristic class $v_2^2$, being in the image of the differential $d^{0,2}_2(\oH_2(A,\Z))$,  vanishes. The only other possible nonzero characteristic class that can occur on the second page of the spectral sequence is $v_2^3$.  But, by Theorem 7.11 of \cite{DP}, it follows that the order of $v_2^3$ is a divisor of one. Hence, it also vanishes. Since, we already know that the Lyndon-Hochschild-Serre  spectral sequence associated to $\Gamma$ collapses with $\Z$-coefficients, we can  conclude that the differential $d^{0,3}_3(\oH_3(A,\Z))=0$ implying that $v_3^3=0$.   Thus, we have shown that all characteristic classes vanish. Therefore, the Lyndon-Hochschild-Serre spectral sequence collapses at $E_2$ for all coefficient modules that have a trivial $L$-action (see \cite[7.13]{DP}).

Since the holonomy representation of the group $1:4\rtimes 4$ decomposes into
a direct sum of $\rho_6$ and the nontrivial one-dimensional representation, by Corollary 4.2 of \cite{Pet}
(see also \cite[7.3-5]{DP}), it follows that the Lyndon-Hochschild-Serre spectral
sequence associated to the group $1:4\rtimes 4$ collapses at  $E_2$ for all
coefficient modules that have a trivial $L$-action. But this is clearly a
contradiction to our computations of the $4$-dimensional integral cohomology
of the group $1:4\rtimes 4$ (see Remark \ref{diff}).
\end{proof}

The second 4-dimensional counterexample to the conjecture
is the crystallographic group of Section \ref{sect:counter} given by the
matrix \eqref{counter:matrix}. We enclose its cohomology groups in Table
\ref{tab:homology} under the number 2.

In dimension 5, there are 95 non-isomorphic
crystallographic groups with cyclic holonomy of split type.
Out of these, 6  do not satisfy Conjecture \ref{conj},
5 of them with holonomy $\Z_4$ and 1 with holonomy $\Z_8$. We list the matrices corresponding to the  their
holonomy generators below and their cohomology groups in Table
\ref{tab:homology} with numbers from 3 to 8.
{\footnotesize
\begin{align*}
3 & \colon
\begin{bmatrix*}[r]
 -1 &  0 &  0 &  0 &  0 \\
  0 & -1 &  0 &  0 &  0 \\
  0 &  0 &  0 &  0 & -1 \\
  0 &  0 &  1 &  0 &  1 \\
  0 &  0 &  0 & -1 &  1
\end{bmatrix*}, &
4 & \colon
	\begin{bmatrix*}[r]
 -1 &  0 &  0 &  0 &  0 \\
  0 &  0 & -1 &  0 &  0 \\
  0 &  1 &  0 &  1 &  1 \\
  0 &  0 &  0 &  0 &  1 \\
  0 &  0 &  0 &  1 &  0
	\end{bmatrix*}\\
5 & \colon
	\begin{bmatrix*}[r]
 -1 &  0 &  0 &  0 &  0 \\
  0 &  1 &  0 &  0 &  0 \\
  0 &  0 &  1 &  1 &  0 \\
  0 &  0 & -1 &  0 & -1 \\
  0 &  0 & -1 &  0 &  0
	\end{bmatrix*},&
6 & \colon
	\begin{bmatrix*}[r]
  1 &  0 &  0 &  0 &  0 \\
  0 &  0 &  1 &  0 &  0 \\
  0 & -1 &  0 & -1 & -1 \\
  0 &  0 &  0 &  0 & -1 \\
  0 &  0 &  0 & -1 &  0
	\end{bmatrix*} \\
7 & \colon
	\begin{bmatrix*}[r]
  1 &  1 &  0 &  0 &  0 \\
 -1 &  0 &  0 & -1 &  0 \\
  0 &  0 &  0 &  0 & -1 \\
 -1 &  0 &  0 &  0 &  0 \\
  0 &  0 & -1 &  0 &  0
	\end{bmatrix*}, &
8 & \colon
	\begin{bmatrix*}[r]
0 & 0 & 0 & 0 & -1 \\
1 & 0 & 0 & 0 & -1 \\
0 & 1 & 0 & 0 & 0 \\
0 & 0 & 1 & 0 & 0 \\
0 & 0 & 0 & 1 & -1
	\end{bmatrix*}
\end{align*}}

\def\changemargin#1#2{\list{}{\rightmargin#2\leftmargin#1}\item[]}
\let\endchangemargin=\endlist
{\begin{table}[h]
	\caption{Cohomology groups of counterexamples up to dimension 5.
}	
	\label{tab:homology}
	\begin{changemargin}{-2.7cm}{1.0cm}
\begin{tabular}{|c|c|c|c|c|c|c|c|c|c|}
	\hline
	 & Type & CARAT name & $\oH^1$ & $\oH^2$ & $\oH^3$ & $\oH^4$ & $\oH^5$ & $\oH^{2k}$ & $\oH^{2k+1}$\\
	\hline
	$1$ & ${4 \rtimes 4}$
& min.27-1.2
		& $\Z$
		& $\Z \oplus \Z_4 \oplus \Z_2$
		& $\Z \oplus \Z_4 \oplus \Z_2$
		& $\Z_4 \oplus \Z_2^3$
		& $\Z_4 \oplus \Z_2^3$
		& $\Z_4 \oplus \Z_2^3$
		& $\Z_4 \oplus \Z_2^3$
		\\
	\hline
$2$ & ${4 \rtimes 4}$
& min.27-1.5
		& $\Z$
		& $\Z \oplus \Z_4 \oplus \Z_2$
		& $\Z \oplus \Z_4 \oplus \Z_2$
		& $\Z_4^2$
		& $\Z_2^4$
		& $\Z_4^2$
		& $\Z_2^4$
		\\
	\hline
%
%numbers in program numbering:29,33,38,41,42,91
%
$3$ & ${5 \rtimes 4}$
& min.81-1.2
		& $\Z$
		& $\Z^2 \oplus \Z_4 \oplus \Z_2^2$
		& $\Z^2 \oplus \Z_4^2 \oplus \Z_2$
		& $\Z \oplus \Z_4^2 \oplus \Z_2^5$
		& $\Z \oplus \Z_4 \oplus \Z_2^6$
		& $\Z_4^2 \oplus \Z_2^6$
		& $\Z_4^2 \oplus \Z_2^6$
		\\
	\hline
$4$ & ${5 \rtimes 4}$
& min.81-1.5
		& $\Z$
		& $\Z^2 \oplus \Z_4 \oplus \Z_2^2$
		& $\Z^2 \oplus \Z_4 \oplus \Z_2^2$
		& $\Z \oplus \Z_4^2 \oplus \Z_2^3$
		& $\Z \oplus \Z_4 \oplus \Z_2^4$
		& $\Z_4^2 \oplus \Z_2^4$
		& $\Z_4^2 \oplus \Z_2^4$
		\\
	\hline
$5$ & ${5 \rtimes 4}$
& min.82-1.3
		& $\Z^2$
		& $\Z^2 \oplus \Z_4 \oplus \Z_2$
		& $\Z^2 \oplus \Z_4^2 \oplus \Z_2^2$
		& $\Z \oplus \Z_4^2 \oplus \Z_2^4$
		& $\Z_4^2 \oplus \Z_2^6$
		& $\Z_4^2 \oplus \Z_2^6$
		& $\Z_4^2 \oplus \Z_2^6$
		\\
	\hline
$6$ & ${5 \rtimes 4}$
& min.82-1.5
		& $\Z^2$
		& $\Z^2 \oplus \Z_4 \oplus \Z_2$
		& $\Z^2 \oplus \Z_4^2 \oplus \Z_2^2$
		& $\Z \oplus \Z_4^3 \oplus \Z_2$
		& $\Z_4^2 \oplus \Z_2^4$
		& $\Z_4^2 \oplus \Z_2^4$
		& $\Z_4^2 \oplus \Z_2^4$
		\\
	\hline
$7$ & ${5 \rtimes 4}$
& min.82-1.7
		& $\Z^2$
		& $\Z^2 \oplus \Z_4$
		& $\Z^2 \oplus \Z_2^4$
		& $\Z \oplus \Z_4^2 \oplus \Z_2^2$
		& $\Z_4 \oplus \Z_2^5$
		& $\Z_4 \oplus \Z_2^5$
		& $\Z_4 \oplus \Z_2^5$
		\\
	\hline
$8$ & ${5 \rtimes 8}$
& min.142-1.2
		& 0
		& $\Z^2 \oplus \Z_8 \oplus \Z_4$
		& $\Z^2$
		& $\Z \oplus \Z_8 \oplus \Z_4 \oplus \Z_2^2$
		& $\Z_2^2$
		& $\Z_8^2 \oplus \Z_2^4$
		& $\Z_2^2$
		\\
	\hline
\end{tabular}
    \end{changemargin}
 \begin{changemargin}{-2.7cm}{-2.7cm}
{\footnotesize   Notation ${d \rtimes n}$ in the column ``Type'' gives the information that the
group is of dimension $d$ and has holonomy group of order $n$.
CARAT name is the name of the group in the classification given in system
CARAT. Note that CARAT uses left action of the holonomy group, thus holonomy
representation has to be transposed before identification in CARAT. }   \end{changemargin}
  \end{table}}
\begin{rem} {Let $\Gamma=L\rtimes G$ be the group $8:5\rtimes 8$ from the table.
We calculate the terms comprising the right hands side of the conjectured isomorphism:
\[
H^i(G, H^j(L, \Z)) =
\begin{cases}
	\Z &  j =0,4, i=0 \\
	\Z^2 &  j =2,3, i=0 \\
	\Z_8 &   j = 0, i \ge 1, 2 \divides i \\
	\Z_4 &  j = 1, 2 \not\divides i \\
	\Z_2^2 &  j=2,3, i \ge 1, 2 \divides i \\
	\Z_4 &  j = 4, i \ge 1, 2 \divides i \\
	\Z_2 &  j = 5, i \ge 1, 2 \divides i \\
	0 & \text{otherwise}
\end{cases}
\]
 to observe that}
%Sheet
%
%\begin{center}
%\begin{tabular}{c|cccccc}
%5 & 0              & $\Z_2$ &                  0 & $\Z_2$ &                  0 & \ldots \\
%4 & $\Z$           &      0 &             $\Z_4$ &      0 &             $\Z_4$ & \ldots \\
%3 & $\Z \oplus \Z$ &      0 & $\Z_2 \oplus \Z_2$ &      0 & $\Z_2 \oplus \Z_2$ & \ldots \\
%2 & $\Z \oplus \Z$ &      0 & $\Z_2 \oplus \Z_2$ &      0 & $\Z_2 \oplus \Z_2$ & \ldots \\
%1 & 0              & $\Z_4$ &                  0 & $\Z_4$ &                  0 & \ldots \\
%0 & $\Z$           &      0 &             $\Z_8$ &      0 &             $\Z_8$ & \ldots \\
%\hline
%j/i & 0  & 1 & 2 & 3 & 4 \\
%\end{tabular}
%\end{center}
 the free ranks and the orders of the maximal finite subgroups of the groups
$H^k(L \rtimes G, \Z)$ and $\bigoplus_{i+j=k} H^i(G, H^j(L, \Z))$ are the
same for every $k$. This means that the Lyndon-Hochschild-Serre spectral sequence
 collapses at $E_2$ but there are extension problems.
\end{rem}

\section{Other examples}

In this section we state the results of our computations for several examples of crystallographic groups of higher
dimensions.

\subsection{Unresolved cases from \cite{AGPP}}
 Several crystallographic groups that were considered in \cite{AGPP} were not known to satisfy
 Conjecture \ref{conj}.
In section 5 of the same paper, the authors studied all crystallographic groups with holonomy
$\Z_4$ of split type whose holonomy representations are indecomposable. Out of
total 9 such groups, there were two 4 dimensional examples, encoded  $\rho_8$
and $\rho_9$, which were not known to satisfy Conjecture \ref{conj}. We verify
that the example of $\rho_8$ satisfies the conjecture. The
example of  $\rho_9$ is the same as the one considered in Section \ref{sect:counter},
so also $2:4\rtimes 4$ in Table \ref{tab:homology}. Hence, it does not satisfy
the conjecture.

In section 6 of  \cite{AGPP}, in relation to certain $6$-dimensional Calabi-Yau toroidal orbifolds arising in string theory,  some crystallographic groups were considered.  It was shown, that out of possible 18 such groups only two, denoted $\Z_8^{(5)}$ and $\Z_{12}^{(6)}$ were not known to admit local compatible actions. So, their cohomology was not computed.
The 5-dimensional group $\Z_8^{(5)}$  is the same as
example $8:5\rtimes 8$ from Table \ref{tab:homology}. So, it does not satisfy Conjecture \ref{conj}.
We show that the 6-dimensional group $\Z_{12}^{(6)}$ also does not satisfy the conjecture. It has holonomy group $G$ of order 12 generated by the matrix:

{\footnotesize
\[
\begin{bmatrix*}[r]
   0 & 0 & 0 & 0 & 0 & -1 \\
   1 & 0 & 0 & 0 & 0 & -1 \\
   0 & 1 & 0 & 0 & 0 & 0 \\
	 0 & 0 & 1 & 0 & 0 & 1 \\
	 0 & 0 & 0 & 1 & 0 & 0 \\
   0 & 0 & 0 & 0 & 1 & -1 \\
\end{bmatrix*}.
\]}
Cohomology groups of the corresponding crystallographic group are as follows:
\begin{eqnarray*}
\mathrm{H}^i(\Z^6 \rtimes G) &=& \left\{ \begin{array}{lll}
	0  & i=1\\
	\Z^3 \oplus \Z_{12} \oplus \Z_3  & i=2\\
	\Z^2  & i=3\\
	\Z^3 \oplus \Z_{12} \oplus \Z_6 \oplus \Z_3^3  & i=4\\
	\Z_2^2  & i=5\\
	\Z \oplus \Z_{12} \oplus \Z_6^2 \oplus \Z_3^5  & i=6\\
	\Z_2^2   & i=2k-1, & k\geq 4\\
	\Z_{12}^2 \oplus \Z_6^2 \oplus \Z_3^5  & i=2k, & k\geq 4.
                   \end{array}\right.
\end{eqnarray*}

\subsection{Cyclic holonomy group of odd non-prime order}

All previous counterexamples to  Conjecture \ref{conj} have holonomy
of order divisible by 4.
We provide counterexample with odd order holonomy  $\Z_9$.

The first occurrence of a crystallographic group with holonomy $\Z_9$ is in dimension 6 and up to an isomorphism, it is the unique one  in this dimension. We verify that this example satisfies  the conjecture.

We find a counterexample of dimension 8 where the holonomy representation is generated by the matrix:
{\footnotesize
\[
\begin{bmatrix*}[r]
-1 &  0 &-1 &-1 &-1 & 0 & 0 & 0 \\
 1 &  1 & 1 & 0 & 1 & 0 & 1 & 0 \\
 0 &  0 & 0 & 0 & 0 & 0 & 1 & 0 \\
 0 &  0 & 0 & 0 & 1 & 0 & 0 & 0 \\
 1 &  0 & 0 & 0 & 0 & 0 &-1 &-1 \\
-1 & -1 & 0 & 0 &-1 &-1 & 0 & 0 \\
-1 & -1 & 0 & 0 & 0 & 0 & 0 & 1 \\
 0 &  1 & 0 & 0 & 0 & 1 & 0 & 0 \\
\end{bmatrix*}.
\]
}by calculating that
\[
\oH^4(\Z^8 \rtimes G, \Z) = \Z^8  \oplus \Z_9^2 \oplus \Z_3^4 \neq
\Z^8 \oplus \Z_9 \oplus \Z_3^6 = \bigoplus_{i+j=4}\oH^i(G, \oH^j(\Z^8, \Z)).
\]

\def\cprime{$'$}

\end{document}